\documentclass[11pt]{article}
\usepackage{amsmath,amssymb}

\newtheorem{propo}{{\bf Proposition}}[section]
\newtheorem{coro}[propo]{{\bf Corollary}}
\newtheorem{lemma}[propo]{{\bf Lemma}} \newtheorem{theor}[propo]{{\bf
Theorem}} \newtheorem{ex}{{\sc Example}}[section]

\newenvironment{proof}{{\bf Proof.}}{$\Box$}

\def\N{{\mathbb N}}
\def\Z{{\mathbb Z}}
\begin{document}

\vspace*{1.0in}

\begin{center} QUASI-IDEALS OF LEIBNIZ ALGEBRAS
\end{center}
\bigskip

\begin{center} DAVID A. TOWERS 
\end{center}
\bigskip

\begin{center} Department of Mathematics and Statistics

Lancaster University

Lancaster LA1 4YF

England

d.towers@lancaster.ac.uk 
\end{center}
\bigskip

\begin{abstract} A subspace $H$ of a Leibniz algebra $L$ is called a {\em quasi-ideal} if $[H,K]+[K,H]\subseteq H+K$ for every subspace $K$ of $L$. They include ideals and subalgebras of codimension one in $L$. Quasi-ideals of Lie algebras were classified in two remarkable papers of Amayo (\cite{amayo} and \cite{amayo2}). The objective here is to extend those results to the larger class of Leibniz algebras, and to classify those Leibniz algebras in which every subalgebra is a quasi-ideal. 
\end{abstract}

\section{Introduction}
\medskip

An algebra $L$ over a field $F$ is called a {\em Leibniz algebra} if, for every $x,y,z \in L$, we have
\[  [x,[y,z]]=[[x,y],z]-[[x,z],y]
\]
In other words the right multiplication operator $R_x : L \rightarrow L : y\mapsto [y,x]$ is a derivation of $L$. As a result such algebras are sometimes called {\it right} Leibniz algebras, and there is a corresponding notion of {\it left} Leibniz algebra. Every Lie algebra is a Leibniz algebra and every Leibniz algebra satisfying $[x,x]=0$ for every element is a Lie algebra. They were introduced in 1965 by Bloh (\cite{bloh}) who called them $D$-algebras, though they attracted more widespread interest, and acquired their current name, through work by Loday and Pirashvili ({\cite{loday1}, \cite{loday2}). They have natural connections to a variety of areas, including algebraic $K$-theory, classical algebraic topology, differential geometry, homological algebra, loop spaces, noncommutative geometry and physics. A number of structural results have been obtained as analogues of corresponding results in Lie algebras. One such is the structure of Leibniz algebras all of whose subalgebras are ideals given by Kurdachenko, Semko and Subbotin in \cite{kss}.
\par

Put $I=$ span$\{x^2:x\in L\}$. Then
\begin{align} [y,x^2]=&[[y,x],x]-[[y,x],x]=0 \hbox{ and } \nonumber \\
[x^2,y]=&[x,[x,y]]+[[x,y],x]=(x+[x,y])^2-x^2-[x,y]^2\in I, \nonumber
\end{align} so $I$ is an ideal of $L$, $[L,I]=0$ and $[x,y]+[y,x]\in I$ for all $x,y\in L$.
In fact, $I$ is the smallest ideal of $L$ such that $L/I$ is a Lie algebra; $L/I$ is sometimes called the {\em liesation} of $L$. 
\par

We define the following series:
\[ L^1=L,L^{k+1}=[L^k,L] \hbox{ and } L^{(1)}=L,L^{(k+1)}=[L^{(k)},L^{(k)}] \hbox{ for all } k=2,3, \ldots
\]
Then $L$ is {\em nilpotent} (resp. {\em solvable}) if $L^n=0$ (resp.$ L^{(n)}=0$) for some $n \in \N$. The {\em nilradical}, $N(L)$, (resp. {\em radical}, $R(L)$) is the largest nilpotent (resp. solvable) ideal of $L$.
\par

A subalgebra Q of a Lie algebra $L$ is said to be a quasi-ideal of $L$ if $[Q,Fx] \subseteq Q+Fx$ for every $x\in L$. Clearly, every ideal and every subalgebra of codimension one in $L$ is a quasi-ideal of $L$. Core-free quasi-ideals of a Lie algebra were completely determined over any field by Amayo (\cite{amayo}) and independently by Gein (\cite{gein}). Amayo also gave an explicit description of core-free subalgebras of codimension one in $L$ in \cite{amayo2}. Our objective here is to produce similar results for Leibniz algebras.
\par

In section 2 we consider some general results on quasi-ideals and subquasi-ideals. In section 3 we produce a classification of core-free quasi-ideals similar to those of Amayo. Finally in section 4 we classify Leibniz algebras in which every subalgebra is a quasi-ideal, thereby generalising the results of \cite{kss}.
\par

Throughout, $L$ will be a (not necessarily finite-dimensional) Leibniz algebra over a field $F$. If $H$ is a subspace of $L$ we denote by $\langle H\rangle$ the subalgebra generated by $H$. Algebra direct sums will be denoted by $\oplus$, whereas direct sums of the vector space structure alone will be denoted by $\dot{+}$.

\section{Quasi-ideals}
Let $H,K$ be subspaces of a Leibniz algebra. We say that $H$ {\em permutes with} $K$ if $[H,K]+[K,H]\subseteq H+K$. A subspace which permutes with every subspace of a Leibniz algebra $L$ is called a {\em quasi-ideal} of $L$; such a subspace is necessarily a subalgebra of $L$. In this section we will establish analogues of some results of Amayo for Lie algebras in \cite[Section 3]{amayo}.
\par

The following lemma will prove useful.

\begin{lemma}\label{l:lemma 0} Let $H$ be a quasi-ideal of $L$, let $h\in H$ and $x\in L$. If $[x,h]\in H$ then $[h,x] \in H$.
\end{lemma}
\begin{proof} Suppose that $[x,h]\in H$. Since $H$ is a quasi-ideal we have that $[h,x]=\lambda x+h_1$ for some $h_1\in H$, $\lambda \in F$. But $[x,h]+[h,x]\in I$, so $[h,x]\in H+I$. It follows that $\lambda x\in H+I$ and hence that $\lambda=0$ or $x\in H+I$. In either case, $[h,x]\in H$.
\end{proof}

\begin{lemma}\label{l:lemma 1} If $H$ is a quasi-ideal of $L$ then $[L,H^2]+[H^2,L]\subseteq H$.
\end{lemma}
\begin{proof} Let $a,b\in H$, $x\in L$. Then
\[ [x,a]=\alpha x+a_1, [x,b]=\beta x+b_1,
\] for some $a_1, a_2, b_1, b_2 \in H$. Hence
\begin{align}
[x,[a,b]] & = [[x,a],b]-[[x,b],a] \nonumber \\
 & = [\alpha x+a_1,b]-[\beta x+b_1,a] \nonumber \\
 & = \alpha(\beta x+b_1)+[a_1,b]-\beta(\alpha x+a_1)-[b_1,a] \nonumber \\
 & = \alpha b_1-\beta a_1+[a_1,b]-[b_1,a]\in H. \nonumber
\end{align}
But now $[[a,b],x]\in H$ also, by Lemma \ref{l:lemma 0}.
\end{proof}
\medskip

We say that $H$ is an {\em m-step subquasi-ideal} of $L$ if there is a chain of subalgebras of $L$,
\[ H=H_m \hspace{.2cm} qu  \hspace{.1cm} H_{m-1}  \hspace{.1cm} qu  \hspace{.1cm} \ldots \hspace{.1cm} qu \hspace{.1cm} H_0 =L, 
\] where each term in the chain is a quasi-ideal of the next. The next few lemmas are established by induction proofs. However, these are not entirely straightforward so the proofs are included.

\begin{lemma}\label{l:lemma 2} If $H$ is an $m$-step subquasi-ideal of $L$ then $$[L,(H^2)^m]+[(H^2)^m,L]\subseteq H,$$
\end{lemma}
\begin{proof} We use induction on $m$. The result is true for $m=1$ by Lemma \ref{l:lemma 1}. Suppose that it holds for $m=t$ and let $H$ be a $(t+1)$-step subquasi-ideal of $L$. Then there is a chain of subalgebras of $L$
\[ H=H_{t+1} \hspace{.2cm} qu  \hspace{.1cm} H_t  \hspace{.1cm} qu  \hspace{.1cm} \ldots \hspace{.1cm} qu \hspace{.1cm} H_0 =L, 
\] where each term in the chain is a quasi-ideal of the next. Clearly, $H_t$ is a $t$-step subquasi-ideal of $L$, $H$ is a $t$-step subquasi-ideal of $H_1$, $H_1$ is a quasi-ideal of $L$ and $H$ is a quasi-ideal of $H_t$, so
\[ [L,(H_t^2)^t]+[(H_t^2)^t,L]\subseteq H_t, \hspace{.5cm}  [H_1,(H^2)^t]+[(H^2)^t,H_1]\subseteq H 
\]
\[ [L,(H_1)^2]+[(H_1)^2,L]\subseteq H_1 \hbox{ and } [H_t,H^2]+[H^2,H_t]\subseteq H.
\]
Hence
\begin{align} & [L,(H^2)^{t+1}]+[(H^2)^{t+1},L] = [L,[(H^2)^t,H^2]]+[[(H^2)^t,H^2],L]  \nonumber \\
& \subseteq [[L,(H^2)^t],H^2]+[[L,H^2],(H^2)^t]+[(H^2)^t,[H^2,L]]+[[(H^2)^t,L],H^2] \nonumber \\
& \subseteq [H_t,H^2]+[H_1,(H^2)^t]+[(H^2)^t,H_1]+[H_t,H^2]\subseteq H \nonumber
\end{align} The result follows.
\end{proof}

\begin{lemma}\label{l:lemma 3} If $H$ is a quasi-ideal of $L$ then $$[L,(H^2)^n]+[(H^2)^n,L]\subseteq (H^2)^{n-1}$$ for all $n\geq 2$.
\end{lemma}
\begin{proof} We use induction on $n$. If $n=2$,
\begin{align}
[L,(H^2)^2]+[(H^2)^2,L] & = [L,[H^2,H^2]]+[[H^2,H^2],L] \nonumber \\
 & \subseteq [[L,H^2],H^2]+[H^2,[H^2,L]]+[[H^2,L],H^2] \nonumber \\
 & \subseteq [H,H^2]+[H^2,H] \subseteq H^2 \nonumber
\end{align}
So suppose the result holds for $n=r$. Then
\begin{align}
& [L,(H^2)^{r+1}]+[(H^2)^{r+1},L] = [L,[(H^2)^r,H^2]]+[[(H^2)^r,H^2],L] \nonumber \\
 & \subseteq [[L,(H^2)^r],H^2]+[[L,H^2],(H^2)^r]+[(H^2)^r,[H^2,L]]+[[(H^2)^r,L],H^2] \nonumber \\
& \subseteq (H^2)^r + [H,(H^2)^r]+[(H^2)^r,H]\subseteq  (H^2)^r,  \nonumber
\end{align} since $(H^2)^r$ is an ideal of $H$.
\end{proof}

\begin{lemma}\label{l:lemma 4} Let $H$ be an $m$-step subquasi-ideal of $L$. Then
\[ [L,(H^2)^{m+n}]+[(H^2)^{m+n},L]\subseteq (H^2)^n
\]
for all $n\geq 1$.
\end{lemma}
\begin{proof} We use a double induction on $m,n$. The case $m=1$ is given by Lemma \ref{l:lemma 3}. So suppose that the result holds whenever $m<r$ where $r>1$, and let $H$ be an $r$-step subquasi-ideal of $L$. Let $n=1$. Then
\begin{align}
& [L,(H^2)^{r+1}]+[(H^2)^{r+1},L] \nonumber \\
& =[L,[(H^2)^r,H^2]]+[[(H^2)^r],H^2],L] \nonumber \\
& \subseteq [[L,(H^2)^r],H^2]+[[L,H^2],(H^2)^r]+[(H^2)^r,[H^2,L]] \nonumber \\
& \hphantom{\subseteq} +[[(H^2)^r,L],H^2] \nonumber \\
& \subseteq [H,H^2]+[H_1,(H^2)^r]+[(H^2)^r,H_1] \subseteq H^2, \nonumber
\end{align} using Lemma \ref{l:lemma 2}, the fact that $H$ is an $(r-1)$-step subquasi-ideal of $H_1$ and that $H^2$ is an ideal of $H$.
\par

So suppose the result holds for $m=r$ and $n=k$. Then
\begin{align}
& [L,(H^2)^{r+k+1}]+[(H^2)^{r+k+1},L] \nonumber \\
& =[L,[(H^2)^{r+k},H^2]]+[[(H^2)^{r+k}],H^2],L] \nonumber \\
& \subseteq [[L,(H^2)^{r+k}],H^2]+[[L,H^2],(H^2)^{r+k}]+[(H^2)^{r+k},[H^2,L]] \nonumber \\
& \hphantom{\subseteq} +[[(H^2)^{r+k},L],H^2] \nonumber \\
& \subseteq [(H^2)^k,H^2]+[H_1,(H^2)^{r+k}]+[(H^2)^{r+k},H_1] \subseteq (H^2)^{k+1}, \nonumber
\end{align} again using that $H$ is an $(r-1)$-step subquasi-ideal of $H_1$.
\end{proof}

\begin{lemma}\label{l:lemma 5} Let $H$ be an $m$-step quasi-ideal of $L$. Then
\[ [L,H^{(m+n+1)}]+[H^{(m+n+1)},L]\subseteq H^{(m+n)}.
\]
\end{lemma}
\begin{proof} This follows from a similar double induction to that used in Lemma \ref{l:lemma 4}.
\end{proof}

\begin{propo} If $H$ is a subquasi-ideal of $L$ then $H^{(\omega)}$ and  $(H^2)^{\omega}$ are characteristic ideals of $L$. In particular, a perfect subquasi-ideal is always an ideal.
\end{propo}
\begin{proof} Let $d$ be a derivation of $L$ and form the semi-direct product $K=L\rtimes Fd$. Then $L$ is an ideal of $K$ and so $H$ is a subquasi-ideal of $K$. Hence there is a finite chain 
\[ H=H_m \hspace{.2cm} qu  \hspace{.1cm} H_{m-1}  \hspace{.1cm} qu  \hspace{.1cm} \ldots \hspace{.1cm} qu \hspace{.1cm} H_0 =K. 
\] Then, by Lemma \ref{l:lemma 4},
\[ [K,(H^2)^{\omega}]\subseteq \bigcap_{n=1}^{\infty} [K,(H^2)^{m+n}]\subseteq \bigcap_{n=1}^{\infty}(H^2)^n=(H^2)^{\omega}.
\]
Also,
\[ [K,H^{(m+n+1)}]\subseteq H^{(m+n)},
\] by Lemma \ref{l:lemma 5}. Similarly for multiplication by $K$ on the right, whence $H^{(\omega)}$ is an ideal of $K$. As $d$ was an arbitrary derivation the result follows.
\end{proof}

\begin{coro} Let $H$ be a finite-dimensional solvable Leibniz algebra. If $H^2$ is not nilpotent then it cannot be embedded as a core-free subquasi-ideal of any Leibniz algebra.
\end{coro}
\begin{proof} If $H^2$ is not nilpotent then $(H^2)^{\omega}$ is a non-trivial ideal of $L$ contained in $H_L$.
\end{proof}
\medskip

We say that $x\in L$ is a {\em left Engel element} of $L$ if, for each $y\in L$ there exists $n=n(y)$ such that $R_x^n(y)=0$.

\begin{lemma} If $H$ is a quasi-ideal of $L$ which is generated by left Engel elements of $L$, then $H$ is an ideal of $L$.
\end{lemma}
\begin{proof} Let  $y\in L$  and  let  $x$  be one  of the  generators  of $H$  which is a  left Engel  element  of $L$.   Then  there  exists  $x_1 \in  H$  and  $\lambda \in F$  such  that $[y, x] = \lambda y + x_1$  and so for every $n$ we have $R_x^n(y)  = \lambda^ny + x_n$,  where $x_n  \in  H$. As  $R_x^m(y)  = 0$  for  some  $m$  we  have  $\lambda^my\in  H$,  and  so  $y \in  H$  or  $\lambda  =  0$;  in either  case  $[y, x] \in  H$. Also, $[x,y]\in H$ by Lemma \ref{l:lemma 0}, so $H$ is an ideal of $L$.
\end{proof}
\medskip

It follows that quasi-ideals of Leibniz algebras in which every element is left Engel, which we'll call Engel Leibniz algebras, or which are locally nilpotent are necessarily ideals. A finite-dimensional Engel Leibniz algebra is nilpotent (see \cite{ao}).

\section{The classification of core-free quasi-ideals}
The Lie algebra $L$ is called {\em almost abelian} if $L=L^2\dot{+}Fa$ with ad $a$ acting as the identity map on the abelian ideal $L^2$. Every subalgebra of an almost abelian Lie algebra is a quasi-ideal. If $F$ has characteristic two, we define the Lie algebra $K_2=Fx+Fy+Fz$ with multiplication $[x,y]=z,[y,z]=y,[z,x]=x$. Then $K_2$ is a simple Lie algebra in which every subalgebra is a 2-step subquasi-ideal, and $Fz$ is a quasi-ideal. The following result extends \cite[Theorem 3.6]{amayo} of Amayo to Leibniz algebras.

\begin{theor}\label{t:qi} Let $H$ be a core-free quasi-ideal of a Leibniz algebra $L$ over a field $F$. The one of the following occurs
\begin{itemize}
\item[(i)] $H=0$.
\item[(ii)] $H$ has codimension one in $L$.
\item[(iii)] $L$ is a Lie algebra which is almost abelian or isomorphic to $K_2$ (see Amayo \cite[Theorem 3.6]{amayo})..
\item[(iv)] $L=I\dot{+}Fh$ where $[x,h]=x$, $[h,x]=[h,h]=0$ for all $x\in I$. In this case the quasi-ideals are precisely $Fh$ and the subspaces of $I$.
\end{itemize}
\end{theor}
\begin{proof}
Let $H$ be a quasi-ideal of $L$ and assume that $H$ is not an ideal of $L$. Suppose further that $H$ has codimension at least two in $L$, so there exist $x,y\in L$ which are independent modulo $H$. We can also assume that $x$ does not idealise $H$. Let $a\in H$. Then
\begin{align}
[x,a]=\lambda_ax+a_x,\hspace{.5cm} &  [y,a]=\lambda_1y+a_y, & [x+y,a]=\lambda_2(x+y)+a_{x+y} \nonumber \\
[a,x]=\mu_ax+ {}_x a,\hspace{.5cm} &  [a,y]=\mu_1y+{}_ya, & [a,x+y]=\mu_2(x+y)+{}_{x+y}a \nonumber
\end{align}
for some $\lambda_a, \lambda_1, \lambda_2, \mu_a, \mu_1, \mu_2 \in F$, $a_x, a_y, a_{x+y}, {}_xa, {}_ya, {}_{x+y}a \in H$.
\par

Now $[x+y,a]=[x,a]+[y,a]$, so $\lambda_a=\lambda_1=\lambda_2$. Similarly $\mu_a=\mu_1=\mu_2$. Hence, for every $z\in L$ we have
\begin{align}
[z,a]=\lambda_az+a_z, \hspace{.5cm} [a,z]=\mu_az+{}_za
\end{align}
Now $[z,a]+[a,z]\in I$ so $(\lambda_a+\mu_a)z\in H+I$. So, if there exists $z\in L$ such that $z\notin H+I$, we have $\mu_a=-\lambda_a$; otherwise $L=H+I$ and $\mu_a=0$.
\par

The map $\theta : H \rightarrow F : a\mapsto \lambda_a$ is a Leibniz homomorphism, and so has kernel $K$ of codimension one in $H$. Thus, $H=Fh+K$, $K$ is an ideal of $H$ and $[K,L]+[L,K]\subseteq H$. We can also choose $a\in H$ such that $\lambda_a=1$.
\par

Let $k\in K$, $z\in L$. If $z\in H$ then $[k,z]=[z,k]\in K$. If $z\notin H$ then, since $H$ has codimension at least two in $L$, we can find $w\in L$ independent of $z$ modulo $H$. As $K$ has codimension one in $H$ we can find $\mu_1, \mu_2, \mu_3, \mu_4 \in F$, $k_1,k_2,k_3,k_4\in K$ such that 
\begin{align}
[z,k]=\mu_1a+k_1, & [w,k]=\mu_2a+k_2  \nonumber \\
[k,z]=\mu_3a+k_3, & [k,w]=\mu_4a+k_4. \nonumber
\end{align} 
\par

Suppose that $L\not = H+I$. Then we have that $[k,[z,w]]\in H$ and
\begin{align}
[k,[z,w]] & = [[k,z],w]-[[k,w],z] \nonumber \\
 & = [\mu_3a+k_3,w]-[\mu_4a+k_4,z] \nonumber \\
 & = -\mu_3w+\mu_3\,{}_wa+[k_3,w]+\mu_4z-\mu_4\,{}_za+[k_4,z], \nonumber
\end{align}
which gives that $\mu_3=\mu_4=0$. Similarly, since $[[z,w],k]\in H$,
\begin{align}
[[z,w],k] & = [[z,[w,k]]+[[z,k],w] \nonumber \\
 & = [z,\mu_2a+k_2]+[\mu_1a+k_1,w] \nonumber \\
 & = \mu_2z+\mu_2a_z+[z,k_2]-\mu_1w+\mu_1\,{}_wa+[k_1,w], \nonumber
\end{align}
which implies that $\mu_1=\mu_2=0$, and $K$ is an ideal of $L$.
\par

If $L=H+I$, then $H\cap I$ is an ideal of $L$, since $I$ is abelian. Hence $H\cap I\subseteq H_L=0$. But then
$[K,I]=0$ and $[I,K]\subseteq H\cap I=0$ so, again, $K$ is an ideal of $L$.
\par

Factor out $H_L$, so assume that $H=Fh$ is a one-dimensional quasi-ideal of $L$ of codimension at least two in $L$. Suppose first that $L=Fh\dot{+}I$.
Then $[h,h]\in Fh\cap I=0$, $[h,x]=0$, $[x,h]= x+\lambda_xh$ for all $x\in I$. But $[x,h]+[h,x]\in I$, so $\lambda_xh\in Fh\cap I=0$, so $L$ has the structure given in (iv) above. In this case it is easy to check that the quasi-ideals are precisely $Fh$ and the subspaces of $I$.
\par

So now assume that $L\not =Fh+I$. Then, for each $z\in L$,
\begin{align}
[z,h]=z+\lambda_zh \hbox{ and } [h,z]=-z+\mu_zh,
\end{align}
Clearly $[h,h]=\lambda h$. But $[h,h]\in I$, so $\lambda^2h=[h,[h,h]]=0$ giving
\begin{align} [h,h]=0
\end{align}
Also,
\begin{align}
0=[h^2,z] & = [h,[h,z]]+[[h,z],h] \nonumber \\
 & = [h,-z+\mu_zh]+[-z+\mu_zh,h] \nonumber \\
 & = z-\mu_zh-z-\lambda_zh, \nonumber
\end{align}
so $\mu_z=-\lambda_z$. Hence, for all $z\in  L$,
\begin{align}
[z,h]=z+\lambda_zh \hbox{ and } [h,z]=-z-\lambda_zh
\end{align}
Let $z,w$ be two elements of $L$. Then
\begin{align}
[[z,w],h] & = [z,w]+\lambda_{[z,w]}h \nonumber \\
 & = [z,[w,h]]+[[z,h],w] \nonumber \\
 & = [z,w+\lambda_wh]+[z+\lambda_zh,w] \nonumber \\
 & = [z,w]+\lambda_wz+\lambda_z\lambda_wh+[z,w]-\lambda_zw-\lambda_z\lambda_wh, \nonumber
\end{align}
so
\begin{align}
[z,w]=\lambda_{[z,w]}h-\lambda_wz+\lambda_zw
\end{align}
Applying $h$ on the right to (5) gives
\begin{align}
[[z,w],h] & =[z,w]+\lambda_{[z,w]}h = -\lambda_w[z,h]+\lambda_z[w,h] \nonumber \\
 & = -\lambda_wz -\lambda_w\lambda_zh+\lambda_zw+\lambda_z\lambda_wh, \nonumber
\end{align}
whence
\begin{align}
[z,w]=-\lambda_{[z,w]}h-\lambda_wz+\lambda_zw
\end{align}
Equations (5) and (6) imply that $2\lambda_{[z,w]}=0$.
\par

It follows from (4) that there is a basis $\{z,x_i \mid i\in J\}$ for $L$ such that
\[ [x_i,h]=x_i, [h,x_i]=-x_i \hbox{ for all } i\in J.
\] (so $\lambda_{x_i}=0$ for $i\in J$). Now (6) gives
\begin{align}
[x_i,x_j]=\lambda_{ij}h
\end{align}
for all $i,j\in J$, where $\lambda_{ij}=\lambda_{[x_i,x_j]}$.
\par

Suppose that the characteristic of $F$ is not $2$. Then $\lambda_{[z,w]}=0$ for all $z,w\in L$, and so $[x_i,x_j]=0$ for all $i,j\in J$. Thus 
$A=\Sigma_{i\in J} Fx_i$ is an abelian ideal of codimension one in $L$. Since $h\notin A$, $L=A\dot{+}Fh$ and $A=L^2$. Moreover, $[a,h]=-[h,a]=a$ for all $a\in A$ so $L$ is a Lie algebra.
\par

So suppose now that the characteristic of $F$ is two. If $A$ has at least three elements then
\begin{align}
0 & = [x_i,[x_j,x_k]]-[[x_i,x_j],x_k]+[[x_i,x_k],x_j] \nonumber \\
 & = [x_i,\lambda_{jk}h]-[\lambda_{ij}h,x_k]+[\lambda_{ik}h,x_j] \nonumber \\
 & = \lambda_{jk}x_i+\lambda_{ij}x_k-\lambda_{ik}x_j. \nonumber
\end{align}
Hence $\lambda_{jk}=\lambda_{ij}=\lambda_{ik}=0$ and we have a Lie algebra again.
\par

Finally, suppose that $F$ has characteristic two and that $A$ is two dimensional. If $x,y\in A$, let $[x,y]=\alpha h$, $[y,x]=\beta h$. Then, since $[x,y]+[y,x]\in I$,
\[ 0=[x,[x,y]+[y,x]]=[x,(\alpha+\beta)h]=(\alpha+\beta)x,
\] so $\alpha=\beta$ and $[x,y]=[y,x]$. Moreover, $[x,h]=[h.x]$ for all $x\in A$, and so $L$ is a Lie algebra. 
\end{proof}
\medskip

We will call the algebras given in Theorem \ref{t:qi} (iv), {\em non-Lie almost abelian} Leibniz algebras. It remains to consider subalgebras of codimention one in $L$. This is done in the following Theorem. For non-Lie Leibniz algebras this turns out to be more straightforward than for Lie algebras.

\begin{theor}\label{t:codim} Let $L$ be a Leibniz algebra with a core-free subalgebra $H$ of codimension 1 in $L$. Then either
\begin{itemize}
\item[(i)] $L$ is a Lie algebra and so is given by \cite[Theorem 3.1 and 4.1]{amayo2}, or
\item[(ii)] $H=Fh$ and $L=Fx+Fh$ where $[x,h]=x$, $[h,x]=[x,x]=[h,h]=0$.
\end{itemize}
\end{theor}
\begin{proof}Let $H$ be a subalgebra of $L$ with codimension one in $L$. If $I\subseteq H$ then $I\subseteq H_L=0$ and $L$ is a Lie algebra.
If $H\subset I$ then $I=L$ which is impossible. So suppose that $L=I+H$. Now $I\cap H$ is an ideal of $L$ since $I$ is an abelian ideal of $L$, so
$I\cap H\subseteq H_L=0$ and $L=I\dot{+}H$. Put $I=Fx$. Then
\[ [h,x]=0, [x,h]=\lambda_hx \hbox{ for all } h\in H.
\]
Thus we can choose a basis $\{h,h_i \mid i\in J\}$ for $H$ such that $[x,h_i]=0$, $[x,h]=x$. Put $K=\sum_{i\in J}Fh_i$. Let $[h,h_i]=\lambda h+\mu k$, where $k\in K$. Then
\[  \lambda x=[x,[h,h_i]]=[[x,h],h_i]-[[x,h_i],h]=0,
\] so $\lambda =0$ and $[h,K]\subseteq K$. Similarly, by considering $[x,[h_i,h]]$, we have that $[K,h]=0$. Also, if $k_1,k_2\in K$, by considering $[x,[k_1,k_2]]$
we have that $[K,K]\subseteq K$. It follows that $K$ is an ideal of $L$ and so $K\subseteq H_L=0$ and $H=Fh$. Now, $[h,h]\in H\cap I=0$ and $[x,x]=0$ since $x\in I$, so we have case (ii).
\end{proof}
\medskip

Note that the algebra in Theorem \ref{t:codim} is a non-Lie almost abelian Leibniz algebra, so no new non-Lie Leibniz algebras appear here. It is a cyclic Leibniz algebra generated by $x+h$.

\section{Leibniz algebras in which every subalgebra is a quasi-ideal}
\medskip

Let ${\mathcal Q}$ denote the set of all Leibniz algebras in which every subalgebra is a quasi-ideal. Then the following is easy to check.

\begin{lemma}\label{l:factor} ${\mathcal Q}$ is factor algebra closed.
\end{lemma}
\medskip

A Lie algebra in which every subalgebra is a quasi-ideal is abelian or almost abelian (see \cite[Theorem 3.8]{amayo}). So, in studying the non-Lie Leibniz algebras in ${\mathcal Q}$, we consider two cases: where $L/I$ is abelian and where $L/I$ is almost abelian. First we need some preliminary results.

\begin{lemma}\label{l:dim} Let $L\in {\mathcal Q}$. Then $\dim I\leq 1$.
\end{lemma}
\begin{proof} By Lemma \ref{l:factor} and \cite[Theorem 3.8]{amayo}, $L/I$ is quasi-abelian. Let $x\in I$, $y\in L$. Then $[x,y]=\alpha x + \beta y\in I$,since $<x>=Fx$ is a quasi-ideal, so $\beta y\in I$. If $y\in I$ then $[x,y]=0$; if $y\notin I$ then $\beta=0$ and $[x,y]=\alpha x$. Hence $[Fx,L]\subseteq Fx$. Also, $[L,Fx]=0$, so every subspace of $I$ is an ideal of $L$. 
\par

Suppose that $I\not = 0$. If every subalgebra of $L$ is an ideal then $L$ is as described in \cite[Theorem A]{kss} and $\dim I=1$. So suppose that $L$ has a subalgebra which is not an ideal of $L$. Then it has a cyclic subalgebra $<x>$ which is not an ideal of $L$. Clearly $<x>=Fx+F[x,x]$ and $<x>_L=F[x,x]$, since $F[x,x]\in I$ and so is an ideal of $L$. It follows that $L/F[x,x]$ is given by Theorem \ref{t:qi} or Theorem \ref{t:codim}. Clearly cases (i) and (iv) of the first of these cannot arise. So either $L/F[x,x]$ is quasi-abelian or is given by Theorem \ref{t:codim} (ii).
\par

Suppose first that the latter holds.Then we have that $L/F[h,h]$ is given by Theorem \ref{t:codim} (ii), using the same notation as there. Clearly $I=Fx+F[h,h]$ and
\[ [x,h]=x+\alpha [h,h],\, [h,x]=[x,x]=0 \hbox{ for some } \alpha \in F.
\]
But now $<x+h>=L$ so $L$ is cyclic. Put $L=Fy+Fy^2+Fy^3$ with $[y^3,y]=\alpha y^2+\beta y^3$. Then $<y^2>=Fy^2$ is a quasi-ideal, and so $[y^2,y]=y^3=\lambda y^2+ \mu y$, which is impossible, so this case can't arise.
\par

So suppose now that $L/F[x,x]$ is quasi-abelian. Then $I\subseteq F[x,x]$ so $I=F[x,x]$ and $\dim I=1$.
\end{proof}

\begin{lemma} Let $L$ be a Leibniz algebra with $\dim I=1$ and $[x,x]\not = 0$ for all $x\in L\setminus I$. Then $I\subseteq Z(L)$.
\end{lemma}
\begin{proof} Let $I=Fa$, $x\notin I$ and suppose that $[x,x]=\lambda a$, $[a,x]=\mu a$ for some $\lambda, \mu \in F$, $\lambda\not =0$. If $\mu\not =0$, then 
\[ \left[x-\frac{\lambda}{\mu}a,x-\frac{\lambda}{\mu}a\right]=\lambda a - \lambda a =0 \hbox{ and } x-\frac{\lambda}{\mu}a\notin I,
\] contradicting the hypothesis.
\end{proof}

\begin{lemma}\label{l:ab} Let $L\in {\mathcal Q}$ and let $L/I$ be abelian with $I\neq 0$. Then one of the following holds
\begin{itemize}
\item[(i)] $[x,x]\neq 0$ for all $x\in L\setminus I$;
\item[(ii)] $L=Fx+Fa$ with $[a,x]=a$ and all other products zero; or
\item[(iii)] $[x,x]=0$ for $x\in L\setminus I$ implies that $x\in Z(L)$.
\end{itemize}
\end{lemma}
\begin{proof} Suppose there exists $x\in L\setminus I$ with $[x,x]=0$, so (i) doesn't hold. Put $I=Fa$. If $L=Fx+I$, the only non-zero product is $[a,x]=\lambda a$. Replacing $x$ by $(1/\lambda)x$ gives the multiplication in (ii).
\par

If $L\neq Fx+I$, choose $y\notin Fx+I$. Then $[y,x]=\lambda a=\alpha y+\beta x$ for some $\lambda, \alpha, \beta \in F$, since $<x>=Fx$ is a quasi-ideal. This implies that $\alpha y\in Fx+I$ and so $\alpha=0$. But now $\beta=0$ and $[y,x]=0$. Suppose that $[a,x]=\mu a$. Then $[y+a,x]=\mu a = \gamma(y+a)+\delta x$ and a similar argument shows that $\mu=\gamma=\delta=0$. It follows that $[Fx+I,x]=0$ and so $[L,x]=0$. 
\par

Now let $y$ be any element of $L$. Then $[x,y]=\lambda a =\alpha x +\beta y$, whence
\[ 0=\lambda [x,a]=\alpha [x,x]+\beta [x,y] = \beta \lambda a.
\]
But then $\beta =0$ which implies that $\alpha=\lambda=0$, or $\lambda a=0$. In either case $[x,y]=0$ and $[x,L]=0$, resulting in case (iii).
\end{proof}
\medskip

We call a Leibniz algebra $L$ {\em extraspecial} if $Z(L)$ is one dimensional and $L/Z(L)$ is an abelian Lie algebra. This class of algebras was introduced in \cite{kss}. They are both right and left Leibniz algebras (sometimes called {\em symmetric} Leibniz algebras).

\begin{theor}\label{t:ab}  Let $L/I$ be abelian with $I\neq 0$. Then $L\in {\mathcal Q}$ if and only if one of the following holds
\begin{itemize}
\item[(i)] $L=Fb+Fa$ where the only non-zero products are $[b,b]=a$, $[a,b]=a$;
\item[(ii)] $L=E\oplus Z$ where $Z\subseteq Z(L)$ and $E$ is an extraspecial subalgebra such that $[x,x]\neq 0$ for every $x\in E\setminus Z(E)$.  
\end{itemize}
\end{theor}
\begin{proof} Let $L\in {\mathcal Q}$. Suppose first that $I\not \subseteq Z(L)$. Put $I=Fa$ and let $R_x\mid_I : L \rightarrow I : a\mapsto [a,x]$. This is a Leibniz homomorphism with kernel $C_L(I)$ and so $C_L(I)$ has codimension one in $L$. Put $L=C_L(I)\dot{+}Fb$. Let $c_1,c_2\in C_L(I)$, $[x,y]=\lambda_{x,y}a$ for all $x,y\in L$ and $[a,b]=\mu a$ where $0\neq \mu\in F$. Then
\[ 0=[c_1,[c_2,b]]=[[c_1,c_2],b]-[[c_1,b],c_2]=\lambda_{c_1,c_2}[a,b]=\lambda_{c_1,c_2}\mu\, a,
\]
so $\lambda_{c_1,c_2}=0$ and $C_L(I)$ is abelian. Put $C_L(I)=C\oplus I$.  Now, $Fa+Fb$ is a two-dimensional non-Lie Leibniz algebra, of which there are only two, both cyclic, with $[b,b]=a$ and $[a,b]=0$ or $a$, the first being nilpotent and the second solvable. Then Lemma \ref{l:ab}(i) implies that $C=0$ and only (i) can hold. If  Lemma \ref{l:ab} (ii) holds than we have case (i) by putting $b=x+a$. If  Lemma \ref{l:ab} (iii) holds, then $C=Z(L)$ and we have a special case of (ii) by putting $E=Fa+Fb$.
\par

Now suppose that $I\subseteq Z(L)$. Then Lemma \ref{l:ab} (ii) cannot occur. Let $U$ be a subalgebra of $L$. If $U\subseteq Z(L)$ it is an ideal of $L$. If $U \not\subseteq Z(L)$ then there exists $u\in U\setminus Z(L)$. Lemma \ref{l:ab} (i) and (ii) imply that $0\neq [u,u]\in U\cap I$, so $I\subseteq U$. But then $U/I$ is an ideal of $L/I$ and so $U$ is an ideal of $L$. It follows that (iii) holds, by \cite[Theorem A]{kss}.
\par

Conversely, if (i) holds then the only subalgebras are $Fa$, $F(b-a)$ and $L$, all of which are quasi-ideals of $L$. so suppose that (ii) holds and let $U$ be a subalgebra of $L$. If $U\subseteq Z(L)$ it is an ideal of $L$, so suppose that $U\not \subseteq Z(L)$. Let $u=x+z \in U\setminus Z(L)$, where $x\in E\setminus Z(E)$. Then $0\neq [u,u]=[x,x] \in U\cap Z(L)$, so $Z(L)\subseteq U$. Let $V$ be a subspace of $L$. Then
\[ \left[\frac{U}{Z(L)},\frac{V+Z(L)}{Z(L)}\right]+\left[\frac{V+Z(L)}{Z(L)},\frac{U}{Z(L)}\right]\subseteq \frac{U+V}{Z(L)},
\] so $[U,V]+[V,U]\subseteq U+V$ and $U$ is a quasi-ideal of $L$. 
\end{proof}
\medskip}

In case (ii) of Theorem \ref{t:ab} every subalgebra of $L$ is an ideal (see \cite{kss}). However, that is not the case for the algebra in (i), since $\langle b-a\rangle=F(b-a)$ is not an ideal. Next we consider the case where $L/I$ is almost abelian.

\begin{theor}\label{t:alab}  Let $L\in {\mathcal Q}$ and let $L/I$ be almost abelian with $I\neq 0$. Then 
\begin{itemize}
\item[(i)] $I\subseteq Z(L)$;
\item[(ii)] $F$ is a non-perfect field of characteristic $2$, $Z(L)=Fz$, there is an ideal $B/Z(L)$ of $L/Z(L)$ of codimension one in $L$ and $L=B\dot{+}Fh$, where $B=C\dot{+}Fz$ is an extraspecial Leibniz algebra with $C$ a subspace of $B$ such that $[c,c']=[c',c]=\lambda_{\{c,c'\}}z$ ($\lambda_{\{c,c'\}}\in F$, $\lambda_c=\lambda_{\{c,c\}}\in F\setminus F^2$), $[c,h]=[h,c]=c$ for all $c\in C$, and $[h,h]=z$.
\end{itemize}
\end{theor}
\begin{proof} There is an ideal $B/I$ of $L/I$ and an element $h+I\in L/I$ such that $[b,h]=b+\alpha_b a$, $[h,b]=-b+\beta_b a$, $[a,b]=\lambda_b a$, $[a,h]=\mu a$ for all $b\in B$, where $\alpha_b, \beta_b, \lambda_b, \mu\in F$ and $I=Fa$. Suppose that $[b,b]=0$ for some $b\in B\setminus I$. Then we must have $\beta_b=0$ since, otherwise $[h,b]$ shows that $<b>=Fb$ is not a quasi-ideal. Moreover, $[h+a,b]=-b+\lambda_b a$ shows that $\lambda_b=0$. But now, $[b+a,b+a]=0$ and $[h,b+a]=-b$ shows that $<b+a>=F(b+a)$ is not a quasi-ideal. Hence 
\begin{align}
[b,b]\neq 0 \hbox{ for all } b\in B\setminus I. 
\end{align}
So let $[b,b]=\gamma a$ where $b\in B\setminus I$ and $0\neq \gamma\in F$. If $\lambda_b\neq 0$ then
\[ \left[b-\frac{\gamma}{\lambda_b}a,b-\frac{\gamma}{\lambda_b }a\right]=0 \hbox{ and } b-\frac{\gamma}{\lambda_b }a\in B\setminus I.
\]
It follows that $[I,B]=0$.
\par

We now employ a similar argument for $Fh+I$. Suppose that $[h+\lambda a,h+ \lambda a]=0$. Then $[h+\lambda a,b]=-b+\beta_b a$ which shows that $\beta_b=0$ since, otherwise, $<h+\lambda a>=F(h+\lambda a)$ is not a quasi-ideal. But now $[h+\lambda a,b+a]=-b$ again shows that $<h+\lambda a>$ is not a quasi-ideal. It follows that $[x,x]\neq 0$ for all $x\in (Fh+I)\setminus I$. If $[h,h]=\delta a$ and $\mu\neq 0$ then
\[ \left[h-\frac{\delta}{\mu} a,h-\frac{\delta}{\mu} a\right]=0 \hbox{ and } h-\frac{\delta}{\mu} a\in (Fh+I)\setminus I.
\]
It follows that $[I,Fh+I]=0$, whence $[I,L]=[L,I]=0$ and $I\subseteq Z(L)$, establishing (i).
\par

Now, if $b\in B\setminus Z(L)$ and $Z(L)=Fz$ then
\[ [h,b]=[h,[b,h]]=[[h,b],h]-[[h,h],b]=-[b,h],
\]
so $\beta_b=-\alpha_b$ and $[h,b]=-b-\alpha_b z$. But then, for all $b_1,b_2\in B\setminus Z(L)$,
\[ 0=[h,[b_1,b_2]]=[[h,b_1],b_2]-[[h,b_2],b_1]=-[b_1,b_2]+[b_2,b_1]
\] so $[b_1,b_2]=[b_2,b_1]$. Moreover,
\[ [b_1,b_2]=[b_1,[b_2,h]]=[[b_1,b_2],h]-[[b_1,h],b_2]=-[b_1,b_2].
\]
If $F$ has characteristic different from $2$ this implies that $[b_1,b_2]=0$. But then $(8)$ yields that $B=Z(L)$ and $L/I$ is not almost abelian.
\par

If $F$ has characteristic $2$, put $[h,h]=z$, let $C'$ be any subspace of $B$ complementary to $Fz$ and let $C=R_h(C')$. Then the products in $L$ are as given in the result. If $\lambda_c = \alpha^2$ for some $\alpha \in F$, then $[c+\alpha h,c+\alpha h]=0$ and $[h,c+\alpha h]$ shows that $<c+\alpha h>=F(c+\alpha h)$ is not a quasi-ideal.
\par

So, suppose that $L$ is as given in the theorem. It is straightforward to check that all of these are Leibniz algebras.  Let $U$ be a subalgebra of $L$. If $U\subseteq Z(L)$ it is an ideal of $L$, so suppose that $U\not \subseteq Z(L)$. Let $u=c+\alpha h+\beta z \in U\setminus Z(L)$, where $c\in C$, $\alpha, \beta \in F$. Then $$0\neq (\lambda_c+\alpha^2)z=[u,u] \in U\cap Z(L),$$ so $Z(L)\subseteq U$. Let $V$ be a subspace of $L$. Then
\[ \left[\frac{U}{Z(L)},\frac{V+Z(L)}{Z(L)}\right]+\left[\frac{V+Z(L)}{Z(L)},\frac{U}{Z(L)}\right]\subseteq \frac{U+V}{Z(L)},
\] so $[U,V]+[V,U]\subseteq U+V$ and $U$ is a quasi-ideal of $L$. 
\end{proof}
\medskip

We'll finish with a simple example of the algebras described in Theorem \ref{t:alab} above.

\begin{ex} Let $F=\Z_2(t)$ and let $L$ have basis $c,z,h$ with $[c,c]=tz$, $[h,h]=z$, $[c,h]=[h,c]=c$ and all other products zero. Then $L$ is a (symmetric) Leibniz algebra. Its only one-dimensional subalgebra is $Fz$, which is an ideal. All other proper non-trivial subalgebras have codimension one in $L$ and so are quasi-ideals of $L$.
\end{ex}

\end{document}